\documentclass[a4paper,12pt]{amsart}
\usepackage{amsmath,amsfonts,amssymb,amsthm,latexsym,amscd}
\usepackage[dvips]{graphicx}
\usepackage[T1]{fontenc}

\usepackage{parskip}

\def\cD{\mathcal{D}}
\def\C{\mathcal{C}}
\def\D{\mathbb{D}}
\def\Diff{\operatorname{Diff}}

\def\Id{\operatorname{Id}}
\def\Ker{\operatorname{Ker}}
\def\Sign{\operatorname{\textbf{Sign}}}

\def\a{\alpha}
\def\b{\beta}
\def\area{\operatorname{area}}
\def\diam{\operatorname{diam}}
\def\ev{\operatorname{ev}}

\def\sign{\operatorname{sign}}
\def\g{\gamma}

\def\s{\sigma}
\def\vol{\operatorname{vol}}

\newtheorem{thm}{Theorem}[section]
\newtheorem{thm*}{Theorem}
\newtheorem{lem}[thm]{Lemma}

\newtheorem*{q*}{Question}

\theoremstyle{definition}

\newtheorem{rem}[thm]{Remark}
\newtheorem*{rem*}{Remark}
\newtheorem*{cor*}{Corollary}

\def\OP{\operatorname}
\def\B{\mathbf}

\begin{document}

\title[Quasi-morphisms and $L^p$-metrics]{Quasi-morphisms and $L^p$-metrics on groups of volume-preserving diffeomorphisms}

\author{Michael Brandenbursky}

\begin{abstract}
Let $M$ be a smooth compact connected oriented manifold of dimension at least two endowed with a volume form $\mu$. We show that every homogeneous quasi-morphism on the identity component $\Diff_0(M,\mu)$ of the group of volume-preserving diffeomorphisms of $M$, which is induced by a quasi-morphism on the fundamental group $\pi_1(M)$, is Lipschitz with respect to the $L^p$-metric on $\Diff_0(M,\mu)$. As a consequence, assuming certain conditions on $\pi_1(M)$, we construct bi-Lipschitz embeddings of finite dimensional vector spaces into $\Diff_0(M,\mu)$.
\end{abstract}

\maketitle

\section{Introduction and main results}

\subsection{The $L^p$-metric}

Let $M$ be a compact connected and oriented Riemannian manifold and
let $\Diff(M,\mu)$ denote the group of smooth diffeomorphisms of
$M$ acting by the identity on a neighborhood of the boundary and
preserving the volume form $\mu$ induced by the metric.
Unless otherwise stated we assume that $\Diff(M,\mu)$ is
equipped with the Whitney $C^{\infty}$-topology.

In the present paper we study the geometry of the identity component
$\Diff_0(M,\mu)$ of the above group endowed with the right
invariant $L^p$-metric. It is defined as follows. Let
$$
\mathcal{L}_p\{g_t\}:=
\int_0^1 dt \left(\int_M|\dot{g_t}(x)|^p\mu \right)^{\frac 1p}
$$
be the $L^p$-length
of a smooth isotopy $\{g_t\}_{t\in [0,1]}\subset\Diff_0(M,\mu)$,
where $|\dot{g_t}(x)|$ denotes the length of the tangent
vector $\dot{g_t}(x)\in T_xM$ induced by the Riemannian
metric. Observe that this length is right-invariant, that is,
$\mathcal{L}_p\{g_t\circ f\}=\mathcal{L}_p\{g_t\}$
for any $f\in \Diff(M,\mu)$. It defines a non-degenerate right-invariant
metric on $\Diff_0(M,\mu)$ by
$$
{\bf d}_p(g_0,g_1):=\inf_{g_t}\mathcal{L}_{p}\{g_t\},
$$
where the infimum is taken over all paths from
$g_0$ to $g_1$. See Arnol'd-Khesin \cite{AK} and Khesin-Wendt \cite[Section~3.6]{KW} for a detailed discussion.

If $p=2$ then the group $\Diff_0(M,\mu)$ is in fact equipped with a
Riemannian metric inducing the above $L^2$-length.  The geodesics of
this metric are the solutions of the equations of the flow of an
incompressible fluid \cite{Ar}, which makes the $p=2$ case the most
interesting. It is known that if $M$ is a simply connected Riemannian
manifold of dimension at least three then the $L^2$-diameter of the
group $\Diff_0(M,\mu)$ is finite \cite{Sh}. On the other hand
Eliashberg and Ratiu \cite{ER} proved that this diameter is infinite
for surfaces and for manifolds with positive first Betti number, and whose fundamental group has a trivial center. In \cite{BK} Kedra and the author showed, that under certain conditions on the fundamental group, the diameter of the identity component of the group of volume-preserving diffeomorphisms is also infinite.

\subsection{Quasi-morphisms on $\Diff_0(M,\mu)$}

Quasi-morphisms are known to be a helpful tool in the
study of algebraic structure of non-Abelian groups, especially the
ones that admit a few or no (linearly independent) real-valued homomorphisms.
Recall that a {\it quasi-morphism} on a group $G$ is a function
$\varphi\colon G \to \B R$ which satisfies the homomorphism equation up to
a bounded error: there exists $K_\varphi > 0$ such that
$$|\varphi (ab) -\varphi(a) -\varphi (b)| \leq K_\varphi$$ for all $a,b
\in G$. The infimum of all such $K_\varphi$ is called the defect of $\varphi$ and is denoted by $D_\varphi$.
A quasi-morphism $\varphi$ is called {\it homogeneous} if we have
$\varphi (a^m) = m \varphi (a)$ for all $a \in G$ and $m \in \B Z$.
Any quasi-morphism $\varphi$ can be {\it homogenized}: setting
\begin{equation}\label{eq:definition-homogenization}
\widetilde{\varphi} (a) := \lim_{k\to +\infty} \varphi (a^k)/k
\end{equation}
we get a homogeneous (possibly trivial) quasi-morphism
$\widetilde{\varphi}$.

\subsubsection{Polterovich construction}

Let $m\in M\setminus\partial M$. Suppose that the center $Z(\pi_1(M,m))$ is trivial and the group $\pi_1(M,m)$ admits a \emph{non-trivial} homogeneous quasi-morphism $\widetilde{\phi}\colon\pi_1(M,m)\to\B R$. For each $x\in M\setminus\partial M$ let us choose an arbitrary geodesic path from $x$ to $m$. In \cite{Polt} Polterovich constructed the induced \emph{non-trivial} homogeneous quasi-morphism $\widetilde{\Phi}$ on $\Diff_0(M,\mu)$ as follows:\\
For each $x\in M$ and an isotopy $\{g_t\}_{t\in[0,1]}$ between $\Id$ and $g$ let $g_x$ be a closed loop in $M$ which is a concatenation of a geodesic
path from $m$ to $x$, the path $g_t(x)$ and a described above geodesic path from $g(x)$ to $m$. Denote by $[g_x]$ the corresponding element in $\pi_1(M,m)$ and set
$$\Phi(g):=\int\limits_M \widetilde{\phi}([g_x])\mu\qquad\qquad \widetilde{\Phi}(g):=\lim\limits_{k\to\infty}\frac{1}{k}\int\limits_M \widetilde{\phi}([(g^k)_x])\mu.$$
The maps $\Phi$ and $\widetilde{\Phi}$ are well-defined quasi-morphisms because the center $Z(\pi_1(M,m))$ is trivial and every diffeomorphism in $\Diff_0(M,\mu)$ is volume-preserving. In addition, the quasi-morphism $\widetilde{\Phi}$ neither depends on the choice of a family of geodesic paths, nor on the choice of a base point $m$. For more details see \cite{Polt}.

\subsubsection{Gambaudo-Ghys construction}

Let $\cD:=\Diff(\D^2,\partial\D^2,\area)$ be the group of smooth area-preserving diffeomorphisms of the unit disc in the Euclidean plane which equal to the identity near the boundary. The group $\cD$ admits a unique
(continuous, in the proper sense) homomorphism to the reals--the famous Calabi homomorphism (see e.g. \cite{Ba,C}). At the same time $\cD$ is known to admit many (linearly independent) homogeneous quasi-morphisms
(see e.g. \cite{Ba-G,BEP,GG}). In what follows we describe a particular geometric construction of such quasi-morphisms, essentially contained in Gambaudo-Ghys \cite{GG} and studied by the author in \cite{B}, which produces quasi-morphisms on $\cD$ from
quasi-morphisms on the pure braid groups $\B P_n$ on $n$ strings.

Denote by $X_n$ the space of all \emph{ordered} $n$-tuples of distinct points in $\D^2$. Let us fix a base point $\overline{z}=(z_1,\ldots,z_n)\in X_n$ and let $\overline{x}=(x_1,\ldots, x_n)$ be any other point in $X_n$. Take  $g\in\cD$ and  any path $g_t$, $0 \leq t\leq 1$, in $\cD$
between $\Id$ and $g$. Connect $\overline{z}$ to $\overline{x}$ by a straight line in $(\D^2)^n$, then act on $\overline{x}$ with the path $g_t$, and then connect $g(\overline{x})$ to $\overline{z}$ by the straight line in $(\D^2)^n$. We get a loop in $(\D^2)^n$. More specifically it looks as follows. Connect $z_i$ to $x_i$ by straight lines $\mathfrak{l}_{1,i}\colon\left[0,\frac{1}{3}\right]\to\D^2$
in the disc, then act with the path $g_{3t-1}$, $\frac{1}{3} \leq t\leq \frac{2}{3}$, on each $x_i$, and then
connect  $g(x_i)$ to $z_i$ by straight lines $\mathfrak{l}_{2,i}\colon\left[\frac{2}{3},1\right]\to\D^2$
in the disc, for all $1\leq i \leq n$. It is easy to show that for almost all $n$-tuples of different points $x_{1},\ldots, x_{n}$ in the disc the concatenations of the paths $\mathfrak{l}_{1,i}\colon\left[0,\frac{1}{3}\right]\to\D^{2}$, $g_{3t-1}\colon\left[\frac{1}{3},\frac{2}{3}\right]\to\D^{2}$ and $\mathfrak{l}_{2,i}\colon\left[\frac{2}{3},1\right]\to\D^{2}$, $i=1,\ldots,n$, yield a loop in $X_n$. The homotopy type of this loop is an element in $\B P_n$ (here $\B P_n$ is identified with the fundamental group $\pi_1 (X_n,\overline{z})$). This element is independent of the choice of $g_t$ because $\cD$ is contractible (see e.g. \cite{FLP,T}), it will be denoted by $\g(g;\overline{x})$. Let $\widetilde{\varphi}_n$ be a homogeneous quasi-morphism on $\B P_n$. Denote $d\overline{x}:=dx_1\cdot\ldots\cdot dx_n$ and set
\begin{equation}
\Phi_n (g):=\int\limits_{X_n}\widetilde{\varphi}_n(\g(g;\overline{x}))d\overline{x} \qquad\qquad \widetilde{\Phi}_n (g):=\lim_{k\to +\infty}\Phi_n (g^k)/k
\end{equation}
The function $\Phi_n$ is a well-defined quasi-morphism on $\cD$ and the function $\widetilde{\Phi}_n$ is a well-defined \emph{homogeneous} quasi-morphism on $\cD$, see \cite{B}.
\begin{rem*}
The group $\B P_2$ is infinite cyclic, hence every homogeneous quasi-morphism $\widetilde{\varphi}_2\colon\B P_2\to\B R$ is a homomorphism. The celebrated theorem of Banyaga \cite{Ba} states that the kernel of the Calabi homomorphism $\C\colon\cD\to\B R$ is a simple group. It follows that $\widetilde{\Phi}_2(g)=K_{\widetilde{\Phi}_2}\cdot\C(g)$ for every $g\in\cD$, where $K_{\widetilde{\Phi}_2}$ is a real constant independent of $g$.
\end{rem*}

\subsection{Main results}

A map $\psi\colon (X_1,\OP{d}_1)\to (X_2,\OP{d}_2)$ between metric spaces
is called {\em Lipschitz} if there exists a constant
$A\geq 0$ such that
$$
\OP{d}_2(\psi(x),\psi(y))\leq A\cdot \OP{d}_1(x,y).
$$
The following theorems are our main technical results. They are proven in Section \ref{S:proofs}.

\begin{thm*}\label{T:lsl-M}
Let $M$ be a compact connected and oriented Riemannian manifold of dimension at least two, such that $Z(\pi_1(M,m))$ is trivial. Let $\widetilde{\phi}\colon\pi_1(M,m)\to\B R$ be a homogeneous quasi-morphism. Then the induced homogeneous quasi-morphism
$$
\widetilde{\Phi}\colon \Diff_0(M,\mu) \to \B R
$$
is Lipschitz with respect to the $L^p$-metric on the group $\Diff_0(M,\mu)$.
\end{thm*}

\begin{rem*}
Theorem \ref{T:lsl-M} shows that the diameter of $(\Diff_0(M,\mu),\B d_p)$ is infinite if $Z(\pi_1(M,m))$ is trivial and $\pi_1(M,m)$ admits a non-trivial homogeneous quasi-morphism.
\end{rem*}

\begin{thm*}\label{T:lsl}
Let $\widetilde{\varphi}_n$ be a homogeneous quasi-morphism on $\B P_n$. Then the induced homogeneous quasi-morphism
$$
\widetilde{\Phi}_n\colon \cD \to \B R
$$
is Lipschitz with respect to the
$L^p$-metric on the group $\cD$.
\end{thm*}

\begin{rem*}
It follows from \cite{B,GG} that for every $n\geq3$ there exists a homogeneous quasi-morphism $\widetilde{\Phi}_n\colon \cD \to \B R$ such that it does not vanish on the kernel of the Calabi homomorphism $\C\colon\cD\to\B R$. Hence Theorem \ref{T:lsl} gives another proof of the Theorem of Eliashberg and Ratiu \cite{ER}, which states that the diameter of $(\Ker(\C),\B d_p)$ is infinite (see also \cite{GL}).
\end{rem*}

\subsection{Applications}

Let $(G,\|\cdot\|_G)$ and $(G',\|\cdot\|_{G'})$ be two normed semigroups. A function
$f\colon G\to G'$
is a \emph{bi-Lipschitz embedding} if it is an \emph{injective homomorphism}, and there exists
a constant $A\geq1$ such that
\begin{equation}\label{eq:q-i-e}
A^{-1}\|g\|_G\leq \|f(g)\|_{G'}\leq A\|g\|_G.
\end{equation}
Note that if $G$ and $G'$ are groups, then the norms $\|\cdot\|_G$ and $\|\cdot\|_{G'}$ define right-invariant metrics on the groups $G$ and $G'$ in a natural way, i.e., $d_G(g,h):=\|gh^{-1}\|_G$ and $d_{G'}(g',h'):=\|g'h'^{-1}\|_{G'}$ for all $g,h\in G$ and $g',h'\in G'$. In this case by definition every bi-Lipschitz embedding is a quasi-isometric embedding.

Recall that the word norm on a group $\Gamma$ generated
by a symmetric finite set $S\subset \Gamma$ is defined by
$$
|\g|_S:=
\min\{k\in \B N\,|\,\g=s_1\ldots s_k \text{ where }s_i\in S\}.
$$
The word metric is defined by $\OP{d}_S(\g_1,\g_2):=|\g_1(\g_2)^{-1}|_S$.
It is right-invariant and it depends on the choice of a finite
generating set up to a bi-Lipschitz equivalence \cite[Example 8.17]{BH}.

In \cite{BK} Kedra and the author showed, that under certain conditions on the fundamental group $\pi_1(M,m)$, the group $(\Diff_0(M,\mu),{\bf d}_p)$ contains quasi-isometrically embedded finitely generated free Abelian group of an arbitrary finite rank. The following theorem generalizes this result for a wide class of compact Riemannian manifolds.

\begin{thm*}\label{T:free-ab}
\textbf{1.} Let $M$ be a compact connected and oriented Riemannian manifold of
dimension at least $3$, such that $Z(\pi_1(M,m))$ is trivial. Let $n\in\B N$ and $\{[\g_i]\}_{i=1}^n$ in $\pi_1(M,m)$. Suppose that $\pi_1(M,m)$ admits a family of homogeneous quasi-morphisms $\{\widetilde{\phi}_i\}_{i=1}^n$, such that $\widetilde{\phi}_i([\g_j])=\delta_{ij}$ where $\delta_{ij}$ is the Kronecker delta. Then $(\Diff_0(M,\mu),{\bf d}_p)$ contains bi-Lipschitz embedded $\B R^n$.\\
\textbf{2.} Let $\Sigma_g$ be a closed orientable surface of genus $g\geq 2$. Then the group $(\Diff_0(\Sigma_g,\mu),{\bf d}_p)$ contains bi-Lipschitz embedded $\B R^g$.
\end{thm*}

Let $M$ be a closed negatively curved Riemannian manifold. Then $\pi_1(M,m)$ has a trivial center and is word-hyperbolic. It follows from \cite[Proposition 3.6]{EF} that for each $n\in\B N$ there exist words $\{[\g_i]\}_{i=1}^n$ in $\pi_1(M,m)$ and a family of homogeneous quasi-morphisms $\{\widetilde{\phi}_i\}_{i=1}^n$, such that $\widetilde{\phi}_i([\g_j])=\delta_{ij}$, where $\delta_{ij}$ is the Kronecker delta. As an immediate corollary we have

\begin{cor*}\it
Let $M$ be a closed negatively curved Riemannian manifold of dimension at least $3$. Then $(\Diff_0(M,\mu),{\bf d}_p)$ contains bi-Lipschitz embedded
vector space of an arbitrary finite dimension.
\end{cor*}

If $(M,\omega)$ is a symplectic manifold, then the group $\Diff_0(M,\mu)$ in all the results above
can be replaced either by the group $\OP{Symp}_0(M,\omega)$ of symplectic diffeomorphisms isotopic to the identity,
or by the group $\OP{Ham}(M,\omega)$ of Hamiltonian diffeomorphisms, see Remark \ref{R:symp}
in the proof of Theorem \ref{T:free-ab}. In case of the group $\OP{Ham}(M,\omega)$ the assumption on the triviality of $Z(\pi_1(M,m))$ may be dropped, i.e., Polterovich quasi-morphisms are well-defined on $\OP{Ham}(M,\omega)$. This follows from the fact that the map $\ev:\OP{Ham}(M,\omega)\to M$, where $\ev(g)=g(m)$, induces a trivial map on $\pi_1(\OP{Ham}(M,\omega),\Id)$, see \cite{McDuff}. The same proof as the proof of Theorem \ref{T:free-ab} (part \textbf{2}) proves the following

\begin{cor*}\it
The group $(\OP{Ham}(\Sigma_g,\omega),{\bf d}_p)$ contains bi-Lipschitz embedded $\B R^g$ for each $g\geq 1$. In particular, the diameter of $(\OP{Ham}(\Sigma_g,\omega),{\bf d}_p)$ is infinite for all $g\geq 1$.
\end{cor*}

\begin{rem*}
The group $\OP{Ham}(M,\omega)$ may be equipped with the famous Hofer metric \cite{Ho,LM}. Similar results to ours with respect to the Hofer metric were obtained by Py and Usher. In \cite{Py} Py showed that for
$g\geq 2$ the group $\OP{Ham}(\Sigma_g,\omega)$ contains bi-Lipschitz embedded copy of an arbitrary finitely generated free Abelian group. Recently Usher \cite{U} generalized this result and showed that for a wide class of symplectic manifolds $\ell^\infty$ bi-Lipschitz embeds into $\OP{Ham}(M,\omega)$.
\end{rem*}

Recall that $\cD:=\Diff(\D^2,\partial\D^2,\area)$ and $\C\colon\cD\to\B R$ is the Calabi homomorphism. In \cite{BG} Benaim and Gambaudo showed that the group $(\Ker(\C),\B d_2)$ contains quasi-isometrically embedded finitely generated free Abelian group of an arbitrary rank. The following theorem generalizes the result above.

\begin{thm*}\label{T:free-ab-D}
For each $n$ the group $(\cD, \B d_p)$ contains bi-Lipschitz embedded $\B R^n$.
Moreover, this statement holds for the group $(\Ker(\C),{\bf d}_p)$.
\end{thm*}

In \cite{CW} Crisp and Wiest generalized the results of \cite{BG} and proved that the group $(\cD,\B d_2)$ contains quasi-isometrically embedded planar right-angled Artin groups. To the best knowledge of the author no similar results are known for infinitely generated groups.

Let $\B Z^\infty$ be a lattice in $\ell^1$, i.e., $\B Z^\infty$ consists of all infinite sequences of integers, such that for each sequence $(n_1,n_2,\ldots)\in\B Z^\infty$ there exists $d\in\B N$ such that $n_i=0$ for each $i>d$. It follows that the metric on $\B Z^\infty$ is the word metric with respect to the infinite set $\{\pm e_i\}_{i=1}^\infty$, where $\pm e_i=(0,\ldots,0,\pm 1,0,\ldots)$ and 1 is placed in the $i$-th entry. The following question was posed to the author by M. Sapir.

\begin{q*}[M. Sapir]\label{q:Sapir}
Does $\B Z^\infty$ quasi-isometrically embed into $(\cD, \B d_p)$?
\end{q*}

Let $\B R^\infty_+$ denote the following positive normed semigroup in $\ell^1$: the semigroup $\B R^\infty_+$ consists of sequences
$(v_1,\ldots,v_k, \ldots)$, where $v_i\geq 0$ and there exists $N>0$ such that $v_i=0$ for each $i\geq N$. The following result is related to the question above and is proven in Section \ref{S:proofs}. It gives an example of a bi-Lipschitz embedding of an infinitely generated semigroup into $(\cD, \B d_p)$.

\begin{thm*}\label{T:undistorted}
The semigroup $\B R^\infty_+$ bi-Lipschitz embeds into $(\cD, \B d_p)$. Moreover, $\B R^\infty_+$ bi-Lipschitz embeds
into $(\Ker(\C), \B d_p)$.
\end{thm*}

\section{Proofs}\label{S:proofs}
The \textit{full braid group} $\B B_n$ on $n$ strings is abstractly defined via the following presentation:
$$\B B_n:=\langle\s_1,\ldots,\s_{n-1}|\hspace{2mm} \s_i\s_j=\s_j\s_i,\hspace{2mm}|i-j|\geq2;\hspace{2mm}\s_i\s_{i+1}\s_i=\s_{i+1}\s_i\s_{i+1}\rangle.$$
For a braid $\g\in \B B_n$ denote by $l(\g)$ the length of $\g$ with respect to the set $\{\s_i\}_{i=1}^{n-1}$.
For each $g\in\Diff_0(M,\mu)$ we denote by $\|g\|_p:=\B d_p(\Id,g)$.

\subsection{Proof of Theorem \ref{T:lsl-M}}
Let $S$ be a finite generating set for the group
$\pi_1(M,m)$, and denote by $\Pi_M\colon {M}_{\bullet}\to M$ the universal
Riemannian covering of $M$. This means that the metric
on $M_{\bullet}$ is induced from the Riemannian metric on $M$. The
corresponding distance will be denoted by $\OP{d}_{\bullet}$.

It is enough to show that $\Phi$ is a large scale Lipschitz map.
This means that we have to show that there exist constants $A,B\geq 0$
independent of $g$ such that
$$
A\cdot \|g\|_p+B\geq|\Phi(g)|.
$$

Let $g\in \Diff_0(M,\mu)$ and $\{g_t\}_{t\in [0,1]}\in \Diff_0(M,\mu)$
be an isotopy from the identity to $g$. It follows from the H\"older
inequality that $\|g\|_p\geq C_p\cdot \|g\|_1$,
where $C_p$ is some positive constant independent of $g$. Hence it is
enough to prove the statement for $p=1$.

For any homogeneous quasi-morphism $\widetilde{\phi}\colon\pi_1(M,m)\to\B R$ we have
\begin{equation}\label{eq:length-inequality-M}
|\widetilde{\phi}(\a)|\leq\left(D_{\widetilde{\phi}}+
\max\limits_{s\in S}|\widetilde{\phi}(s)|\right)\|\a\|_S.
\end{equation}
It follows that
\begin{equation}\label{eq:q-l-inequality}
\left|\Phi(g)\right|\leq K\int_M\|[g_x]\|_S\mu,
\end{equation}
where $K=D_{\widetilde{\phi}}+\max\limits_{s\in S}|\widetilde{\phi}(s)|$.

Recall that the loop $g_x$ is a concatenation of a geodesic path from $m$ to $x$, the path
$\{g_t(x)\}$ and a geodesic path from $g(x)$ to $m$. Let ${m}_{\bullet}\in \Pi_M^{-1}(m)$ and
let $\{g_{\bullet,t}(m_{\bullet})\}$ be the lift of the loop $g_x$ starting at the point $m_{\bullet}$.
The manifold $M$ is compact, hence by the $\check{\textrm{S}}$varc-Milnor
lemma \cite{BH,Mil}, the inclusion of the orbit of $m_{\bullet}$
with respect to the deck transformation group $\pi_1(M,m)$ defines
a quasi-isometry
\begin{equation*}
\pi_1(M,m)\stackrel{q.i.}\simeq ({M}_{\bullet},\OP{d}_{\bullet}).
\end{equation*}
In particular, it means that there exist positive constants $A',B'$, such that
\begin{equation}\label{eq:q-i-Milnor}
{d}_{\bullet}({m}_{\bullet},{g}_{\bullet,1}({m}_{\bullet}))\geq A'\|[g_x]\|_S-B'.
\end{equation}
Denote by $\diam(M)$ the diameter of $M$. We also have the estimate
\begin{equation}\label{eq:diam}
{d}_{\bullet}({m}_{\bullet},{g}_{\bullet,1}({m}_{\bullet}))\leq 2\diam(M)+\int_0^1|\dot{g}_t(x)|dt.
\end{equation}

Combining inequalities \eqref{eq:length-inequality-M}, \eqref{eq:q-l-inequality}, \eqref{eq:q-i-Milnor} and \eqref{eq:diam} we get that
\begin{eqnarray*}
\left|\Phi(g)\right|&\leq&K(A')^{-1}\left(\left(\int_0^1 dt\int_M|\dot{g}_t(x)|\mu\right)+\vol(M)(2\diam(M)+B')\right)\\
&=&K(A')^{-1}\mathcal{L}_1(\{g_t\})+K(A')^{-1}\cdot\vol(M)(2\diam(M)+B').
\end{eqnarray*}
Since the above inequalities hold for any isotopy $\{g_t\}_{t\in[0,1]}$ between the identity and $g$,
we obtain that
\begin{equation*}
|\Phi(g)|\leq A\cdot \|g\|_p+B,
\end{equation*}
where $A=C_p\cdot K(A')^{-1}$ and $B=K(A')^{-1}\cdot\vol(M)(2\diam(M)+B')$
and this concludes the proof.
\qed

\subsection{Proof of Theorem \ref{T:lsl}}\label{SS:proof_lsl} Let $n\geq 2$. It is enough to show that the \emph{non-homogeneous} quasi-morphism $\Phi_n\colon\cD\to\B R$ is large scale Lipschitz, i.e., there exist two constants $A,B\geq 0$, such that for every $g\in\cD$
$$|\Phi_n(g)|\leq A\|g\|_p+B.$$

Let $g\in\cD$. For an isotopy $\{g_t\}\in\cD$ between $Id$
and $g$, any $\overline{x}\in X_n$ and $1\leq i,j\leq n,\thinspace i\neq j$ let $l_{i,j}\colon[0,1]\to \B S^1$, such that
$$l_{i,j}(t):=\frac{g_t(x_i)-g_t(x_j)}{\|g_t(x_i)-g_t(x_j)\|}\quad \textrm{and} \quad  L_{i,j}(\overline{x}):=\frac{1}{2\pi} \int\limits_{0}^{1}\left\|\frac{\partial}{\partial t}(l_{i,j}(t))\right\|dt,$$
where $\|\cdot\|$ is the Euclidean norm. Note that $L_{i,j}(\overline{x})$
is the length of the path $l_{i,j}(t)$ divided by $2\pi$. It follows that $L_{i,j}(\overline{x})+4$ is an upper bound for the number of times the string $i$ turns around the string $j$ in the positive direction plus the number of times the string $i$ turns around the string $j$ in the negative direction in the braid $\g(g;\overline{x})$. Recall that a representative of the braid $\g(g;\overline{x})$ is build using any isotopy $\{g_t\}\in\cD$ between $Id$ and $g$. It follows that the number of crossings in any such representative is less then or equal to $\sum_{i<j}^{n}2\left(L_{i,j}(\overline{x})+4\right)$. By definition the number of crossings in any such representative of the braid $\g(g;\overline{x})$ is bigger than the length of the braid $\g(g;\overline{x})$. Thus we get the following inequality
\begin{equation}\label{eq:inequality1}
\displaystyle\sum_{i<j}^{n}2\left(L_{i,j}(\overline{x})+4\right)\geq l(\g(g;\overline{x})),
\end{equation}
where $l(\g(g;\overline{x}))$ is the word length of the braid $\g(g;\overline{x})$.
Take any finite generating set $S$ of $\B P_n$. Note that for any homogeneous quasi-morphism
$\widetilde{\varphi}_n\colon\B P_n\to\B R$ one has
\begin{equation}\label{eq:length-inequality}
|\widetilde{\varphi}_n(\g)|\leq\left(D_{\widetilde{\varphi}_n}+
\max\limits_{s\in S}|\widetilde{\varphi}_n(s)|\right)l_S(\g),
\end{equation}
where $l_S(\g)$ is the length of a word $\g$ with respect to $S$, and $D_{\widetilde{\varphi}_n}$ is the defect of ${\widetilde{\varphi}_n}$. Recall that the pure braid group $\B P_n$ is a normal subgroup of finite index in $\B B_n$. It follows from \cite[Corollary 24]{Harpe} that there exist two positive constants $K_{1,S}$ and $K_{2,S}$ independent of $\g$,
such that $$l_S(\g)\leq K_{1,S}\cdot l(\g)+K_{2,S}.$$

It follows from \eqref{eq:length-inequality} that
\begin{equation}\label{eq:inequality2}
|\widetilde{\varphi}_n(\g(g;\overline{x}))|\leq N_1l(\g(g;\overline{x}))+N_2,
\end{equation}
where $N_1 =K_{1,S}(D_{\widetilde{\varphi}_n}+\max\limits_{s\in S}|\widetilde{\varphi}_n(s)|)$ and
$N_2=K_{2,S}(D_{\widetilde{\varphi}_n}+\max\limits_{s\in S}|\widetilde{\varphi}_n(s)|)$. Inequalities
\eqref{eq:inequality1} and \eqref{eq:inequality2} yield the
following inequality:
\begin{equation*}
|\widetilde{\varphi}_n(\g(g;\overline{x}))|\leq 2N_1\left(\displaystyle\sum_{i<j}^{n}L_{i,j}(\overline{x})+4\right)+N_2.
\end{equation*}
It follows that
\begin{equation}\label{eq:inequality3}
\left|\Phi_n(g)\right|\leq N_3\left(\sum\limits_{i<j}^n
\int\limits_{\D^2\times\D^2}L_{i,j}(\overline{x}) dx_idx_j\right)+B,
\end{equation}
where $N_3=2N_1 \cdot\vol((\D^2)^{n-2})$ and $B=(4N_1(n-1)n+N_2)\vol((\D^2)^n)$.
It follows from the definition of $L_{i,j}$ that
\begin{equation}\label{eq:inequality4}
\sum\limits_{i<j}^n\int\limits_{\D^2\times\D^2}L_{i,j}(\overline{x}) dx_idx_j=\frac{(n-1)n}{4\pi}\int\limits_{0}^{1}\int\limits_{\D^2\times\D^2} \left\|\frac{\partial}{\partial t}\left(\frac{g_t(x)-g_t(y)}{\|g_t(x)-g_t(y)\|}\right)\right\|dx dy dt.
\end{equation}
Cauchy-Schwartz inequality yields
\begin{equation}\label{eq:inequality5}
\int\limits_{0}^{1}\int\limits_{\D^2\times\D^2} \left\|\frac{\partial}
{\partial t}\left(\frac{g_t(x)-g_t(y)}{\|g_t(x)-g_t(y)\|}\right)\right\|dx dy dt\leq
\int\limits_{0}^{1}\int\limits_{\D^2\times\D^2}\frac{4\|\dot{g}_t(x)\|}{\|g_t(x)-g_t(y)\|}dx dy dt.
\end{equation}
By using the polar coordinates we conclude that for each $x\in \D^2$
\begin{equation*}
\int\limits_{\D^2}\frac{1}{\|x-y\|}dy\leq 4\pi.
\end{equation*}
Using the above inequality and the fact that the isotopy $g_t$ is area-preserving we have
\begin{equation}\label{eq:inequality6}
\int\limits_0^1\int\limits_{\D^2\times\D^2}\frac{\|\dot{g}_t(x)\|}{\|g_t(x)-g_t(y)\|}dy dx dt\leq
4\pi\int\limits_0^1\int\limits_{\D^2}\|\dot{g}_t(x)\|dx dt.
\end{equation}
We combine inequalities \eqref{eq:inequality3}, \eqref{eq:inequality4}, \eqref{eq:inequality5}, \eqref{eq:inequality6} and get
$$\left|\Phi_n(g)\right|\leq A\int\limits_0^1\int\limits_{\D^2}\|\dot{g}_t(x)\|dx dt+B:=A\mathcal{L}_1\{g_t\}+B,$$
where $A=4N_3(n-1)n$.
Since the above inequality holds for any isotopy $\{g_t\}_{t\in[0,1]}$ between the identity and $g$, we obtain that
$$\left|\Phi_n(g)\right|\leq A\|g\|_1+B.$$
The above inequality concludes the proof of the theorem in case $p=1$.

Let $p>1$. It follows from H\"older inequality that there exists a positive constant $C_p$ such that
$\|g\|_1\leq C_p\|g\|_p$, and the proof follows.
\qed

\subsection{Proof of Theorem \ref{T:free-ab}}
Suppose that the dimension of $M$ is at least $3$. Hence there exists a family of disjoint simple closed curves $\{[\g'_i]\}_{i=1}^n$ such that the curve $\g'_i$ is free loop homotopic to the curve $\g_i$. It follows from the tubular neighborhood theorem that for each $\g'_i$ there exists
$r_i>0$, the standard $(n-1)$-dimensional ball $B_{r_i}^{n-1} \subset \B{R}^n$ of radius $r_i>0$,
and a volume-preserving embedding
$$
\OP{emb}_i:B_{r_i}^{n-1}\times \B S^1\hookrightarrow M,
$$
such that $\OP{emb}_i|_{\{0\}\times \B S^1}$ is the curve $\g'_i$ and $r_i=\OP{d}_M(\g'_i, \OP{emb}_i|_{\partial B_{r_i}^{n-1}\times \B S^1})$. Moreover, if $i\neq j$ then the images of $\OP{emb}_i$ and $\OP{emb}_j$ are disjoint.
The volume form on the product $B_{r_i}^{n-1}\times \B S^1$ is the standard Euclidean volume and $\OP{d}_M$ denotes the distance on
$M$ induced by the Riemannian metric.

Let $r=\min\limits_{1\leq i\leq k}{r_i}$. Then $\OP{emb}_i:B_r^{n-1}\times \B S^1\hookrightarrow M$
is volume-preserving for each $i$ and $\OP{emb}_i|_{{0}\times  \B S^1}=\g'_i$.
It is straightforward to construct a smooth isotopy of volume-preserving diffeomorphisms
$$
g_t:B_{r}^{n-1}\times \B S^1\to B_{r}^{n-1}\times \B S^1
$$
with $g_0=\OP{Id}$ such that:
\begin{itemize}
\item
For each $t\in \B R$ the diffeomorphism $g_t$ equals to the identity
in the neighborhood of $\partial B_{r}^{n-1}\times \B S^1$, and
the time-one map $g_1$ is equal to the identity on $B_{r'}^{n-1}\times \B S^1$,
where $0<r'<r$.
\item
Each diffeomorphism $g_t$ preserves the foliation of
$B_r^{n-1}\times \B S^1$ by the circles $\{x\}\times \B S^1$. Each diffeomorphism $g_t$ preserves the orientation for $t\geq 0$. In addition, for
every $x\in B^{n-1}_{r'}$ the restriction $g_t\colon \{x\}\times \B S^1\to \{x\}\times \B S^1$ is the rotation by $2\pi\,t$.
\item
For all $s,t\in\B R$ we have $g_{t+s}=g_t\circ g_s$.
\end{itemize}

\begin{rem}\label{R:symp}
Notice that if $M$ is a symplectic manifold then the
above isotopies can be constructed to be Hamiltonian.
\end{rem}

We identify $B^{n-1}_{r}\times \B S^1$ with its image with respect to the embedding $\OP{emb}_i$. Then we extend
an isotopy $g_t$ by the identity on $M\setminus (B_r^{n-1}\times~\B S^1)$ obtaining smooth isotopies $g_{t,i}\in \Diff_0(M,\mu)$.
Using the fact that $\widetilde{\phi}_i([\g_j])=\delta_{ij}$ we get
\begin{equation}\label{eq:Phi>0}
\widetilde{\Phi}_i(g_{1,i}):=\lim\limits_{k\to\infty}\frac{1}{k}\int\limits_M \widetilde{\phi}_i([(g_{1,i}^k)_x])\mu\geq
\vol(B_{r'}^{n-1}\times \B S^1)>0,
\end{equation}
where the first inequality follows from the fact that $\lim\limits_{k\to\infty}\widetilde{\phi}_i([(g_{1,i}^k)_x])$ is zero for $x\in M\setminus(B_{r}^{n-1}\times \B S^1)$, non-negative for $x\in (B_{r}^{n-1}\times \B S^1)$ and equals to 1 for $x\in (B_{r'}^{n-1}\times \B S^1)$.
Similarly, in case when $j\neq i$ we have
\begin{equation}\label{eq:Phi=0}
\widetilde{\Phi}_i(g_{1,j}):=\lim\limits_{k\to\infty}\frac{1}{k}\int\limits_M \widetilde{\phi}_i([(g_{1,j}^k)_x])\mu=0.
\end{equation}
We define a homomorphism $\Psi\colon\B R^n\to \Diff_0(M,\mu)$ by setting
$$\Psi(v):=g_{v_1,1}\circ\ldots\circ g_{v_n,n},$$
where $v=(v_1,\ldots,v_n)\in\B R^n$.
Diffeomorphisms $g_{v_i,i}$ and $g_{v_j,j}$ have disjoint supports for $i\neq j$ and commute for all $1\leq i,j\leq n$. It follows that $\Psi$ is a monomorphism, and in addition $\widetilde{\Phi}_i(g_{v_i,j})=v_i\widetilde{\Phi}_i(g_{1,j})$ for all $1\leq i,j\leq n$, because $\widetilde{\Phi}_i$ is a continuous homogeneous quasi-morphism.

We claim that $\Psi$ is a bi-Lipschitz embedding. Indeed, by Theorem \ref{T:lsl-M} and by \eqref{eq:Phi>0} and \eqref{eq:Phi=0} the following inequalities hold for each $1\leq i\leq n$:
$$\|g_{v_1,1}\circ\ldots\circ g_{v_n,n}\|_p\geq A^{-1}\left|\widetilde{\Phi}_i(g_{v_1,1}\circ\ldots\circ g_{v_n,n})\right|=
A^{-1}\cdot |v_i|\left|\widetilde{\Phi}_i(g_{1,i})\right|,$$
where $A$ is the maximum over the Lipschitz constants of the functions ${\widetilde{\Phi}}_i\colon\Diff_0(M,\mu)\to\B R$. It follows that
$$\|g_{v_1,1}\circ\ldots\circ g_{v_n,n}\|_p\geq \left((n\cdot A)^{-1}\min_i\left|\widetilde{\Phi}_i(g_{1,i})\right|\right) \|v\|,$$
where $\|v\|$ is the $l^1$-norm of $v$ in $\B R^n$. On the other hand
$$\|g_{v_1,1}\circ\ldots\circ g_{v_n,n}\|_p\leq\sum\limits_{i=1}^n\mathcal{L}_p\{g_{t\cdot v_i,i}\}=
\sum\limits_{i=1}^n |v_i|\mathcal{L}_p\{g_{t,i}\}\leq\max_i\mathcal{L}_p\{g_{t,i}\}\|v\|,$$
where $\{g_{t,i}\}$ is the isotopy between $\Id$ and $g_{1,i}$ and $t\in[0,1]$.

Now we deal with the case when $M=\Sigma_g$ and $g\geq 2$. Let us take a basis $\{[\a_i],[\b_i]\}_{i=1}^g$ of
$H_1(\Sigma_g, \B Z)\cong \B Z^{2g}$ such that all $\a_i$, $\b_i$ are simple closed curves, the oriented intersection number
$\#([\a_i]\cap [\b_i])=1$, the intersection $\a_i\cap\b_i$ consists of 1 point, $\a_i\cap\b_j=\emptyset$, $\a_i\cap\a_j=\emptyset$ and $\b_i\cap\b_j=\emptyset$ for all different $1\leq i,j\leq g$. Denote by
$\rho_i\colon H_1(\Sigma_g, \B Z)\to \B R$ the homomorphism which sends $[\a_i]$ to $1$ and all other elements to zero.
Let $\widetilde{\phi}_i\colon \pi_1(\Sigma_g, m)\to\B R$ be the composition of $\rho_i$ with the projection homomorphism
$\pi_1(\Sigma_g, m)\to H_1(\Sigma_g, \B Z)$.

The same proof as in the previous step shows that there exists a family of smooth isotopies $g_{t,i}$ in $\Diff_0(\Sigma_g,\mu)$ such that
\begin{itemize}
\item
For each $t\in\B R$ diffeomorphisms $g_{t,i}$ have disjoint supports for $1\leq i\leq g$.
\item
For all $s,t\in\B R$ and $1\leq i\leq g$ we have $g_{t+s,i}=g_{t,i}\circ g_{s,i}$.
\item
For each induced homomorphism $\Phi_i\colon\Diff_0(\Sigma_g,\mu)\to\B R$ we have $\Phi_i(g_{1,i})>0$ and $\Phi_i(g_{1,j})=0$
for all $1\leq i,j\leq n$.
\end{itemize}
Let $\Psi\colon\B R^g\to \Diff_0(\Sigma_g,\mu)$ be a homomorphism such that
$$\Psi(v):=g_{v_1,1}\circ\ldots\circ g_{v_g,g},$$
where $v=(v_1,\ldots,v_g)$. The same proof as in the previous step shows that $\Psi$ is a bi-Lipschitz embedding.
\qed

\subsection{Proof of Theorem \ref{T:free-ab-D}}
Let $\widetilde{\sign}_n\colon\B P_n\to\B R$ be the homogeneous quasi-morphism defined by signature link invariant $\sign$ \cite{L} of links in $\B R^3$. For a precise definition and properties of $\widetilde{\sign}_n$ see \cite{B,GG}. Denote by $\widetilde{\Sign}_n$ the induced homogeneous quasi-morphism on $\cD$. Let $H\colon\D^2\to\B R$ be a smooth function, such that $H(x)=h(\|x\|^2)$, where $h=0$ in the neighborhood of $0$ and $1$, and $\int\limits_{\D^2}H(x) dx=0$. Denote by $g_t$ the flow generated by $H$. Then $g_t\in\Ker(\C)$, because $\C(g_t)=2t\int\limits_{\D^2}H(x) dx$ (see \cite{GG1}). In \cite{GG} it was shown that there exists $K_n>0$, such that
$$\widetilde{\Sign}_n(g_t)=K_n\cdot t\int_0^1y^{n-2}h(y)dy.$$

Let $n\in \B N$. It is straight forward to construct a family of functions $\{H_i\}_{i=1}^n$, where $H_i(x)=h_i(\|x\|^2)$, such that
\begin{itemize}
\item
Each Hamiltonian flow $g_{t,i}$ generated by $H_i$ lies in $\Ker(\C)$.
\item
Diffeomorphisms $g_{t,i}$ and $g_{s,j}$ commute for $s,t\in \B R$, $1\leq i,j\leq n$.
\item
The matrix $\left(
              \begin{array}{ccc}
                \widetilde{\Sign}_3(g_{1,1}) & \cdots & \widetilde{\Sign}_3(g_{1,n}) \\
                \vdots & \vdots & \vdots \\
                \widetilde{\Sign}_{n+2}(g_{1,1}) & \cdots & \widetilde{\Sign}_{n+2}(g_{1,n}) \\
              \end{array}
            \right)
$ is non-singular.
\end{itemize}
It follows that there exists a family $\{\widetilde{\Phi}_i\}_{i=1}^n$ of homogeneous quasi-morphisms on $\cD$, such that $\widetilde{\Phi}_i$ is a linear combination of $\widetilde{\Sign}_n$'s and
\begin{equation}\label{eq:sign-basis}
\widetilde{\Phi}_i(g_{t,j})=\left\{
                                          \begin{array}{c}\begin{aligned}
                                            &t  &\rm{if}&\quad i=j\\
                                            &0  &\rm{if}&\quad i\neq j\\
                                            \end{aligned}
                                          \end{array}
                                        \right ..
\end{equation}

Let $\Psi\colon\B R^n\to \cD$ be a map, such that
$\Psi(v):=g_{v_1,1}\circ\ldots\circ g_{v_n,n}$
and $v=(v_1,\ldots,v_n)$. It follows from the construction of $\{g_{v_i,i}\}_{i=1}^n$ that $\Psi$ is a monomorphism. The same proof as in Theorem \ref{T:free-ab} shows that there exists $A_n>0$, such that
$$\|g_{v_1,1}\circ\ldots\circ g_{v_n,n}\|_p\leq A_n \|v\|.$$
Let $1\leq i\leq n$. All diffeomorphisms $g_{v_1,1},\ldots, g_{v_n,n}$ pair-wise commute. Hence by Theorem \ref{T:lsl} and equality \eqref{eq:sign-basis} we have
$$\|g_{v_1,1}\circ\ldots\circ g_{v_n,n}\|_p\geq A^{-1}\left|\widetilde{\Phi}_i(g_{v_1,1}\circ\ldots\circ g_{v_n,n})\right|=
A^{-1}\cdot |v_i|\left|\widetilde{\Phi}_i(g_{1,i})\right|,$$
where $A$ is the maximum over the Lipschitz constants of the functions ${\widetilde{\Phi}}_i\colon\cD\to\B R$. It follows that
$$\|g_{v_1,1}\circ\ldots\circ g_{v_n,n}\|_p\geq \left((n\cdot A)^{-1}\min_i\left|\widetilde{\Phi}_i(g_{1,i})\right|\right) \|v\|,$$
and the proof follows.
\qed

\subsection{Proof of Theorem \ref{T:undistorted}}\label{SS:proof_undistorted}

\begin{lem}\label{lem:N-inf}
Suppose that there exists a homogeneous quasi-morphism $\widetilde{\Phi}\colon\cD\to\B R$ and a family of diffeomorphisms $\{g_{t,i}\}_{i=1}^\infty$ in $\Ker(\C)$ for each $t\in\B R$ such that
\begin{itemize}
\item
The map $\B R^\infty_+\to \Ker(\C)$, $(v_1,\ldots,v_k,0\ldots)\to g_{v_1,1}\circ\ldots\circ g_{v_1,k}$, is an injective homomorphism for each $k$;
\item
the numbers $\left\{\widetilde{\Phi}(g_{1,i})\right\}_{i=1}^\infty$ have the same sign for all $i\in\B N$;
\item
$\inf\limits_i\left|\widetilde{\Phi}(g_{1,i})\right|\geq M_1$, $\sup\limits_i\mathcal{L}_p\{g_{t,i}\}\leq M_2$, where $\{g_{t,i}\}_{t\in[0,1]}$, and $M_1>0$ and $M_2> 0$.
\end{itemize}
Let $\Psi\colon\B R^\infty_+\to\Ker(\C)$ where $\Psi(v_1,\ldots,v_k,0\ldots)=g_{v_1,1}\circ\ldots\circ g_{v_k,k}$. Then $\Psi$ is a bi-Lipschitz embedding.
\end{lem}

\begin{proof}
Let $(v_1,\ldots,v_k,0\ldots)\in\B R^\infty_+$. It follows from the triangle inequality and the hypothesis that
\begin{equation}\label{eq:right}
\|g_{v_1,1}\circ\ldots\circ g_{v_k,k}\|_p\leq\max\limits_{1\leq i\leq k}\mathcal{L}_p\{g_{t,i}\}\sum_{i=1}^k v_i\leq M_2\sum_{i=1}^k v_i.
\end{equation}
By Theorem \ref{T:lsl} there exists a positive constant $A$ such that
\begin{equation*}
\|g_{v_1,1}\circ\ldots\circ g_{v_k,k}\|_p\geq A^{-1}\left|\widetilde{\Phi}_n(g_{v_1,1}\circ\ldots\circ g_{v_k,k})\right|.
\end{equation*}
Note that $\widetilde{\Phi}(g_{v_1,1}\circ\ldots\circ g_{v_k,k})=\sum_{i=1}^k v_i\widetilde{\Phi}(g_{v_i,i})$ because diffeomorphisms $\{g_{v_i,i}\}$ and $\{g_{v_j,j}\}$ pair-wise commute. By hypothesis the numbers $\widetilde{\Phi}(g_{1,i})$ have the same sign for all $i\in\B N$. Hence
\begin{equation}\label{eq:left}
\|g_{v_1,1}\circ\ldots\circ g_{v_k,k}\|_p\geq A^{-1}\min\limits_{1\leq i\leq k}\left|\widetilde{\Phi}(g_{1,i})\right|\sum_{i=1}^k v_i\geq
M_1\cdot A^{-1}\sum_{i=1}^k v_i.
\end{equation}
Inequalities \eqref{eq:right} and \eqref{eq:left} conclude the proof of the lemma.
\end{proof}

Let us finish the proof of the theorem by constructing a family of diffeomorphisms $\{g_{t,i}\}_{i=1}^\infty\in\Ker(\C)$ as in Lemma \ref{lem:N-inf}.
Let $\{h_s\}_{s\in\left[\frac{1}{4},\frac{1}{3}\right]}$ be a family of $C^\infty$ functions from the interval $[0,1]$ to $\B R$ such that:
\begin{itemize}
\item
the support of $h_s$ is $\left[\frac{1}{2}-s,\frac{1}{2}+s\right]$;
\item
$h_s|_{\left(\frac{1}{2}-s,\frac{1}{2}\right)}$ is positive and $h_s|_{\left(\frac{1}{2},\frac{1}{2}+s\right)}$ is negative;
\item
If $s\neq s'$ then $h_s|_{\left[\frac{3}{8},\frac{5}{8}\right]}\equiv h_{s'}|_{\left[\frac{3}{8},\frac{5}{8}\right]}$. If $s>s'$ and
$y\in \left(\frac{3}{8},\frac{1}{2}\right)$ then $h_s(y)\geq h_{s'}(y)$, if $s>s'$ and $y\in\left(\frac{1}{2},\frac{5}{8}\right)$
then $h_s(y)\leq h_{s'}(y)$;
\item
For each $s\in\left[\frac{1}{4},\frac{1}{3}\right]$ and $p\in\B N$ the integral $\int_0^1 h_s(y) dy=0$, the integral $\int_0^1 yh_s(y) dy<0$, and both functions $s\to \int_0^1 yh_s(y) dy$ and $s\to \int_0^1 y^{\frac{p}{2}}\left|\frac{\partial}{\partial y}h_s(y)\right|^p dy$ are continuous.
\end{itemize}
An example of such functions is shown in Figure \ref{fig:graph}.

\begin{figure}[htb]
\centerline{\includegraphics[height=2.5in]{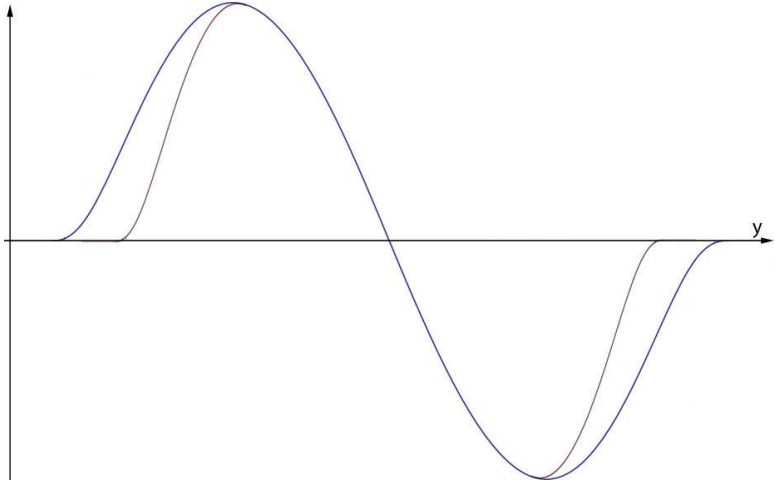}}
\caption{\label{fig:graph} Functions $h_{\frac{1}{4}}(y)$ and $h_{\frac{1}{3}}(y)$ shown in purple and blue respectively.}
\end{figure}

Let us pick an arbitrary discrete set $\{s_i\}_{i=1}^\infty$ in $\left(\frac{1}{4},\frac{1}{3}\right)$. For each $i$ we define a
function $G_i\colon\D^2\to\B R$ such that $G_i(x)=h_{s_i}(\|x\|^2)$. We denote by $g_{t,i}$ the flow generated by the Hamiltonian $G_i$. Recall that $\widetilde{\sign}_n\colon\B P_n\to\B R$ is the homogeneous quasi-morphism defined by classical signature link invariant $\sign$, and  $\widetilde{\Sign}_n$ is the induced homogeneous quasi-morphism on $\cD$. It is left to show that the family $\{g_{t,i}\}_{i=1}^\infty$ satisfies conditions of Lemma \ref{lem:N-inf} for each $t\in \B R$.

\textbf{1.} Each diffeomorphism $g_{t,i}$ lies in $\Ker\C$ because $\C(g_{t,i})=2t\int\limits_{\D^2}G_i(x) dx$ (see \cite{GG1}), and
$\int\limits_{\D^2}G_i(x) dx=\pi\int_0^1 h_{s_i}(y)dy=0$. A straightforward computation shows that in polar coordinates we have
$$g_{t,i}(r,\theta)=\left(r,\theta+2t\frac{\partial}{\partial (r^2)}h_{s_i}(r^2)\right).$$
Hence the diffeomorphisms $g_{t,i}$ and $g_{t',j}$ commute for $t,t'\in \B R$ and for $i,j\in \B N$. It follows that the map
$\B R^\infty_+\to \Ker(\C)$ defined by $(v_1,\ldots,v_k,0\ldots)\to g_{v_1,1}\circ\ldots\circ g_{v_1,k}$,
is an injective homomorphism.

\textbf{2.} Let $n=3$. There exists a positive constant $K_3$ such that
$$\widetilde{\Sign}_3(g_{t,i})=K_3\cdot t\int_0^1 (y+1)h_{s_i}(y)dy$$
for each $i$, see \cite{B,GG}. Recall that $g_{t,i}\in\Ker\C$, hence
$$\widetilde{\Sign}_3(g_{t,i})=K_3\cdot t\int_0^1 yh_{s_i}(y)dy.$$
By construction of functions $h_{s_i}$ we have $\int_0^1 yh_{s_i}(y)dy<0$ for each $i$. Hence the numbers $\left\{\widetilde{\Sign}_3(g_{1,i})\right\}_{i=1}^\infty$ have the negative sign for all $i$.

\textbf{3.} Recall that for each $s\in\left[\frac{1}{4},\frac{1}{3}\right]$ the function
$s\mapsto \int_0^1 yh_s(y) dy$ is continuous and $\int_0^1 yh_s(y) dy<0$. It follows that there exists a constant $M_1>0$ such that
$\left|\int_0^1 yh_s(y) dy\right|\geq M_1K_3^{-1}$ for each $s$. Hence
$$\inf_i\left|\widetilde{\Sign}_3(g_{1,i})\right|\geq M_1.$$

\textbf{4.} Let $p\in\B N$. For each $i$ we have
\begin{eqnarray*}
\|(g_{1,i})\|_p\leq\mathcal{L}_p\{g_{t,i}\}&=&\int_0^1 dt \left(\int_{\D}\|\dot{g}_{t,i}(x)\|^p dx \right)^{\frac{1}{p}}\\
&=&2\pi^{\frac{1}{p}}\left(\int_0^1y^{\frac{p}{2}}\left|\frac{\partial}{\partial y}h_{s_i}(y)\right|^pdy\right)^\frac{1}{p}.
\end{eqnarray*}

Recall that for each $s\in\left[\frac{1}{4},\frac{1}{3}\right]$ the function
$s\mapsto \int_0^1y^{\frac{p}{2}}\left|\frac{\partial}{\partial y}h_s(y)\right|^pdy$ is continuous. It follows that there exists a constant $M_2>0$ such that
$$\left(2^p\pi\int_0^1y^{\frac{p}{2}}\left|\frac{\partial}{\partial y}h_s(y)\right|^pdy\right)^\frac{1}{p}\leq M_2$$
for each $s$. Hence
$$\sup_i\mathcal{L}_p\{g_{t,i}\}\leq M_2.$$
This concludes the proof of the theorem.
\qed

\subsection*{Acknowledgments} The author would like to thank Remi Coulon, Dongping Zhuang and Mark Sapir for helpful conversations. We would like to thank the referee for useful comments which helped us to improve the presentation of this paper.

\bibliographystyle{alpha}

\vspace{3mm}

Department of Mathematics, Vanderbilt University, Nashville, TN 37240\\
\emph{E-mail address:} \verb"michael.brandenbursky@vanderbilt.edu"

\end{document}